\documentclass[11pt, twoside]{article}
\usepackage{latexsym}
\usepackage{amsmath}
\usepackage{amssymb}
\usepackage[all]{xy}
\usepackage{amsfonts}
\usepackage{verbatim}
\usepackage{amsthm}
\usepackage{indentfirst}
\usepackage{mathrsfs}\usepackage{array}
\usepackage{epsfig}
\usepackage{xy}
\usepackage{tensor}
\usepackage{array}
\usepackage{stmaryrd}
\usepackage{graphicx,color}
\usepackage{xcolor}
\usepackage[colorlinks=true,linkcolor=blue,citecolor=blue]{hyperref}
\usepackage{tikz}
\usetikzlibrary{arrows,calc}
\usepackage{etex}
\usepackage{mathdots}
\usepackage{float}
\usepackage{graphics}
\usepackage{pdflscape}
\usepackage{extarrows}
\usepackage{anysize,hyperref}
\input xypic
\xyoption{all}
\usepackage{bm}
\usepackage[perpage,symbol]{footmisc}
\usepackage{setspace}
\SelectTips{cm}{10}

\def\dr{\ar@{->}[r]}
\def\Ker{\mbox{\rm Ker}\,}

\begin{document}
\baselineskip=15pt
\title{\Large{\bf  Injectivity of modules over trusses
\footnotetext{$^\ast$Corresponding author.  ~Yongduo Wang is supported by the National Natural Science Foundation of China (Grant No. 10861143). ~Jian He is supported by Youth Science and Technology Foundation of Gansu Provincial (Grant No. 23JRRA825). ~Dejun Wu is supported by the National Natural Science Foundation of China (Grant No. 12261056).} }}
\medskip
\author{Yongduo Wang$^\ast$,~Shujuan Han,~Dengke Jia,~Jian He and Dejun Wu}
\date{}
\maketitle
\def\blue{\color{blue}}
\def\red{\color{red}}

\newtheorem{theorem}{Theorem}[section]
\newtheorem{lemma}[theorem]{Lemma}
\newtheorem{corollary}[theorem]{Corollary}
\newtheorem{proposition}[theorem]{Proposition}
\newtheorem{conjecture}{Conjecture}
\theoremstyle{definition}
\newtheorem{definition}[theorem]{Definition}
\newtheorem{question}[theorem]{Question}
\newtheorem{notation}[theorem]{Notation}
\newtheorem{remark}[theorem]{Remark}
\newtheorem{remark*}[]{Remark}
\newtheorem{example}[theorem]{Example}
\newtheorem{example*}[]{Example}

\newtheorem{construction}[theorem]{Construction}
\newtheorem{construction*}[]{Construction}

\newtheorem{assumption}[theorem]{Assumption}
\newtheorem{assumption*}[]{Assumption}

\baselineskip=17pt
\parindent=0.5cm

\begin{abstract}
\noindent As the dual notion of projective modules over trusses, injective modules over trusses are introduced. The Schanuel Lemmas on projective and injective modules over trusses are exhibited in this paper.\\[0.3cm]
\textbf{Keywords:} truss; Schanuel Lemma, projectivity, injectivity\\[0.1cm]
\textbf{2020 Mathematics Subject Classification:} 18G80; 18E10
\medskip
\end{abstract}
\medskip

\pagestyle{myheadings}
\markboth{\rightline {\scriptsize Yongduo Wang,~Shujuan Han,~Dengke Jia,~Jian He and Dejun Wu}}
         {\leftline{\scriptsize  Injectivity of modules over trusses}}
\section{Introduction}
The notion of heaps was introduced by H. Pr$\ddot{u}$fer \cite{hp}, R. Baer \cite{rb} and A. K. Su$\check{s}$kevi$\check{c}$ \cite{aks} in the 1920s. A heap is a set $H$ together with a ternary operation
$[---]: H\times H\times H\rightarrow H$ which is associative and satisfies the Mal'cev identities, that is,
\begin{equation*}
[[a,b,c],d,e] = [a,b,[c,d,e]] \qquad \text{and} \qquad [a,b,b] = a = [b,b,a]
\end{equation*}
for all $a,b,c,d,e\in H$. It exhibits that there is a deep connection between groups and heaps.

In 2019, trusses were introduced by T. Brzezi$\acute{n}$ski in \cite{tb} as structures describing two different distributive laws: the well$-$known ring distributivity and the one coming from the recently introduced braces, which are gaining popularity due to their roles in the study of the set$-$theoretic solutions of the Yang$-$Baxter equation. The brace distributive law appeared earlier in the context of quasi$-$rings of radical rings (see \cite{ch}). It turns out that rings and braces can be described elegantly by switching the group structure to a heap structure. This leads to the definition of a truss, which is a set $T$ with a ternary operation $[-, -, -]$ and a binary multiplication $\cdot$ satisfying some conditions, the crucial one being the generalisation of ring and brace distributivity: $a\cdot [b, c, d] = [a\cdot b, a\cdot c, a\cdot d]$ and $[b, c, d]\cdot a = [b\cdot a, c\cdot a, d\cdot a]$, for all $a, b, c, d\in T$. Due to this, we can jointly approach brace and ring theory.

A truss can be understood as a ring in which the Abelian group of addition has no specified neutral element. Also, every truss $T$ is a congruence class of a ring $R(T)$, the universal extension of $T$ into a ring (see \cite{rtb}). Trusses, even though close to rings, differ significantly as the category of trusses has no zero object. It is well known that, for a ring $R$, we must study modules over it, so it is natural to ask: what is the theory of modules over trusses? The notion of modules over trusses was posed and basic properties of it were given by T. Brzezi$\acute{n}$ski (see \cite{tb}). In recent years, modules over trusses were studied by S. Breaz, T. Brzezi$\acute{n}$ski, B. Rybolowicz and P. Saracco from different aspects (see \cite{stbp,stbp1,tb1,tb2,tbp}). In \cite{tbp}, T. Brzezi$\acute{n}$ski, B. Rybolowicz and P. Saracco gave the concept of projective modules over trusses. As the dual notion of projective modules over trusses, injective modules over trusses are introduced. The Schanuel Lemmas on projective and injective modules over trusses are exhibited in this paper.

\section{Preliminaries}
A heap is a set $H$ together with a ternary operation
$[---]: H\times H\times H\rightarrow H$ which is associative and satisfies the Mal'cev identities, that is,
\begin{equation*}
[[a,b,c],d,e] = [a,b,[c,d,e]] \qquad \text{and} \qquad [a,b,b] = a = [b,b,a]
\end{equation*}
for all $a,b,c,d,e\in H$.
A heap $H$ is said to be abelian if for all $a,b,c\in H$, $[a,b,c] = [c,b,a]$.

A heap morphism from $(H,[---])$ to $(H^\prime,[---])$ is a function $f:H\rightarrow H^\prime$ respecting the ternary operations, i.e., such that for all $x$, $y$, $z\in H$, $f([x,y,z])=[f(x),f(y),f(z)]$.
The category of heaps is denoted by Heap and the category of abelian heaps is denoted by Ah. A singleton set $\{\ast\}$  with the unique heap operation $[\ast,\ast,\ast]=\ast$, it is the terminal objective in the category of heaps, we denote it by $\star$. The empty set is the initial object. There is no zero objective in the category of heaps.

With every group $G$ we can associate a heap $H(G)=(G,[-,-,-])$ where $[a,b,c]=ab^{-1}c$ for all $a, b, c\in G$ and every morphism of group is automatically a morphism of heaps. With every non-empty heap $H$ and for a fixed $e\in H$, we can associate a group $G(H;e)$ and the binary operation is $a\cdot b=[a,e,b]$ for all $a, b\in H$. The inverse of $a\in G(H;e)$ is $a^{-1}=[e,a,e]$.

\begin{lemma}\label{le.se}{\rm \cite[Lemma 2.3]{tb1}}
Let $(H,[-,-,-])$ be a heap.
\begin{enumerate}
\item[{\rm (1)}] If $e$, $x$, $y\in H$ are such that $[x,y,e]=e$ or $[e,x,y]=e$, then $x=y$.
\item[{\rm (2)}] For all $v$, $w$, $x$, $y$, $z\in H$,
$$[v,w,[x,y,z]]=[v,[y,x,w],z].$$
\item[{\rm (3)}] For all $x$, $y$, $z\in H$,
$$[x,y,[y,x,z]]=[[z,x,y],y,x]=[x,[y,z,x],y]=z.$$
In particular, in the expression $[x,y,z]=w$, any three elements determine the fourth one.
\item[{\rm (4)}] If $H$ is abelian, then, for all $x_i$, $y_i$, $z_i\in H$, $i=1,2,3$,
$$[[x_1,x_2,x_3],[y_1,y_2,y_3],[z_1,z_2,z_3]]=[[x_1,y_1,z_1],[x_2,y_2,z_2],[x_3,y_3,z_3]].$$
\end{enumerate}
\end{lemma}

A subset $S$ of a heap $H$ that is closed under the heap operation is called a sub-heap of $H$. Every non-empty sub-heap $S$ of an abelian heap $H$ defines a congruence relation $\sim_S$ on $H$:
\begin{equation*}
a\sim_S b \quad \iff \quad \exists~s\in S,\ [a,b,s]\in S \quad \iff \quad \forall~s\in S,\ [a,b,s]\in S.
\end{equation*}
The equivalence classes of $\sim_S$ form an abelian heap with operation induced from that in $H$. Namely,
$
[\bar a, \bar b, \bar c] = \overline{[a,b,c]}
$, where $\bar x$ denotes the class of $x$ in $H/\sim_S$ for all $x\in H$. This is known as the {\em quotient heap} and it is denoted by $H/S$.
For any $s\in S$, the class of $s$ is equal to $S$.

If $\varphi: H\rightarrow K$ is a morphism of abelian heaps, then for all $e\in {\rm Im}\varphi$, the set
\begin{equation*}
\ker_e \varphi:=  \{ a\in H\; |\; \varphi(a)=e\}
\end{equation*}
is a sub-heap of $K$. Different choices of $e$ yielding an isomorphic as heaps and the quotient heap $H/\ker_e \varphi$ does not depend on the choice of $e$. Moreover, the sub-heap relation $\sim_{\ker_e}\varphi$ is the same as the kernel relation defined by: $a\,\Ker\varphi\,b$ if and only if $\varphi(a) = \varphi(b)$. Thus we write $\Ker\varphi$ for $\ker_e \varphi$ and we refer to it as the kernel of $\varphi$.

\begin{definition}{\rm \cite[Definition 3.1]{tb1}}
A truss is an algebraic system consisting of a set $T$, a ternary operation $[-,-,-]$ making $T$ into an Abelian heap, and an associative binary operation $\cdot$ which distributes over $[-,-,-]$, that is, for all $w$, $x$, $y$, $z\in T$,
\begin{equation*}
w[x,y,z]=[wx,wy,wz], \quad  [x,y,z]w=[xw,yw,zw].
\end{equation*}
A truss is said to be commutative(abelian) if the binary operation $\cdot$ is commutative.
\end{definition}
A heap homomorphism between two trusses is a truss homomorphism if it respects multiplications. The category of trusses and their morphisms is denoted by ${\rm Trs}$.

Let $T$ be a truss. A left $T$-module is an abelian heap $M$ together with an associative left action $\lambda_M: T\times M \rightarrow M$ of $T$ on $M$ that distributes over the heap operation. The action is denoted on elements by $t\cdot m = \lambda_M(t,m)$, with $t\in T$ and $m\in M$. Explicitly, the axioms of an action state that, for all $t,t',t''\in T$ and $m,m',m''\in M$,
\begin{subequations}
\begin{equation*}
t\cdot(t'\cdot m) = (tt')\cdot m,
\end{equation*}
\begin{equation*}
 [t,t',t'']\cdot m = [t\cdot m,t'\cdot m,t''\cdot m] ,
\end{equation*}
\begin{equation*}
 t\cdot [m,m',m''] = [t\cdot m,t\cdot m',t\cdot m''].
\end{equation*}
\end{subequations}
If $T$ is a unital  truss and the action satisfies $1\cdot m = m$, then we say that $M$ is a unital or normalised module. A submodule of a left $T$-module $M$ is a subset that is closed both under the heap operation and the action $\cdot$.

A module homomorphism is a homomorphism of heaps between two modules that also respects the actions. As it is customary in ring theory we often refer to homomorphisms of $T$-modules as to $T$-linear maps or morphisms.
The category of left $T$-modules is denoted by $T$-${\rm mod}$, that of left unital $T$-modules  by $ T_{1}$-${\rm mod}$. The terminal heap $\star$ and initial heap $\varnothing$, with the unique possible actions, are the terminal and the initial object in $T$-${\rm mod}$. It should be noted that, since $\star\neq\varnothing$, $T$-${\rm mod}$ do not have zero object.

An element $e$ of a left $T$-module $M$ is called an absorber provided that
\begin{equation*}
t\cdot e = e, \qquad \mbox{for all $t\in T$}.
\end{equation*}
The set of all absorbers in $M$  is denoted by ${\rm Abs}(M) = \{ m\in M\mid t\cdot m = m, \forall\,t\in T\}$.

\begin{proposition}{\rm \cite[Proposition 2.6]{tbp}}
Every epimorphism of $T$-modules is surjective.
\end{proposition}

\begin{proposition}{\rm \cite[Proposition 2.8]{tbp}}
Every monomorphism of $T$-modules is injective.
\end{proposition}

\begin{definition}{\rm \cite[Definition 2.5]{stbp1}}
Let $M$ be a non-empty left $T$-module. For every $e \in M$, the action $\cdot_e\colon  T \times M \rightarrow M$, given by
\[t\cdot_e m = [t\cdot m,t\cdot e,e], \qquad \textrm{for all }m\in M, t \in T,\]
is called the \emph{$e$-induced action} or the \emph{$e$-induced module structure} on $M$ and denote it by $M^{(e)}$. We say that a subset $N\subseteq M$ is {\em an induced submodule} of $M$ if $N$ is a non-empty sub-heap of $M$ and $t\cdot_e n \in N$ for all $t\in T$ and $n,e\in N$.
\end{definition}
Different choices of $e$ yield an isomorphic induced modules and an iteration of an induced action gives an induced action. For all $T$-module morphisms $\varphi:M\rightarrow N$, this yields an analogue of the fist isomorphism theorem for $T$-module: $M/\Ker\varphi\cong {\rm Im}\varphi$.

If $R$ is a ring then we can consider its associated truss ${\rm T}(R)=({\rm H}(R,+),\cdot)$. Moreover, any $R$-module $M$ gives rise, in the same way, to a (unital) ${\rm T}(R)$-module ${\rm T}(M)=({\rm H}(M,+),\cdot)$, whose underlying abelian heap structure is induced by the abelian group one. This assignment gives rise to a functor
$${\rm T} : R\mbox{-}{\rm mod}\longrightarrow {\rm T}(R)\mbox{-}{\rm mod}, \quad (M,+,\cdot) \longmapsto (H(M,+),\cdot), \quad f \longmapsto  f.$$

Let $T$ be a truss (not necessarily unital) and let $\star$ denote the singleton $T$-module. We say that a sequence of non-empty $T$-modules $\xymatrix{M\ar[r]^{f} & N\ar[r]^{g} & P}$  is exact provided there exists $e\in {\rm Im}g$ such that ${\rm Im}f=\ker_{e}g$ as sets. Furthermore, if $e\in Abs(P)$, we say that the sequence of non-empty $T$-modules $\xymatrix{M\ar[r]^{f} & N\ar[r]^{g} & P}$  is Abs-exact

\begin{lemma}\label{lem.abs}{\rm \cite[Lemma 6.1]{tbp}}
Let $M,N,P$ be $T$-modules and $f:M \longrightarrow N$ and  $g:N \longrightarrow P$ be $T$-linear maps. There exist exact sequences
\begin{equation*}
\xymatrix{M\ar[r]^f & N \ar[r]^{g} & P}, \quad \xymatrix{\star \ar[r] & M^{(e)} \ar[r]^{f} & N^{(f(e))}} \quad \mbox{and}\quad  \xymatrix{N \ar[r]^{g} & P \ar[r] & \star}
\end{equation*}
if and only if
\begin{enumerate}
\item[{\rm (1)}]
$f$ is injective and
\item[{\rm (2)}]  $N/{\rm Im}f\cong P$ as $T$-modules,
\end{enumerate}
where the module structure on $N/{\rm Im}f$ is the one for which the canonical projection\\
$\pi:N \longrightarrow N/{\rm Im}f$ is $T$-linear.
\end{lemma}

By abuse of terminology, we will say that
\[
\xymatrix{
\star \ar@{.>}[r] & M \ar@<+0.3ex>@{.>}[r]^{f} \ar@<-0.3ex>[r] & N \ar[r]^g & P \ar[r] & \star
}
\]
is a short exact sequence of $T$-modules to mean that there exists $e\in M$ such that all three sequences in Lemma 2.6 are exact.

\begin{lemma}{\rm \cite[Lemma 4.10]{tb1}}
Let $T$ be a truss, and $M$, $N$ left $T$-modules. Then $M\times N$ is a $T$-module with the product heap and module structures, i.e.\
\begin{enumerate}
\item[{\rm (1)}] with the heap operation defined by
$$
\left[(m_1,n_1), (m_2,n_2), (m_3,n_3)\right] = \left([m_1,m_2,m_3],[n_1,n_2,n_3]\right),
$$
for all $m_1,m_2,m_3\in M$, $n_1,n_2,n_3 \in N$;
\item[{\rm (2)}]for all $x\in T$, $m\in M$ and $n\in N$,
$$
x\cdot (m,n) = (x\cdot m, x\cdot n).
$$
\end{enumerate}
\end{lemma}

\begin{lemma}\label{aa}{\rm \cite[Lemma 4.13]{tb1}}
The set ${\rm Hom}_{T}(M,N)$ is a heap with the point wise heap operation.
\end{lemma}

\begin{corollary}{\rm \cite[Corollary 2.13]{tb1}} \label {ua} A heap homomorphism $\varphi$ is injective if and only if there an element of the codomain with a singleton pre-image, if and only if $ker(\varphi)$ is a singleton (trivial) heap.
\end{corollary}

\section{Projective modules over trusses}

In \cite{tbp}, T. Brzezi$\acute{n}$ski, B. Rybolowicz and P. Saracco gave the concept of projective modules over trusses. The Schanuel Lemma on projective modules over trusses is exhibited in this section.

\begin{definition}{\rm \cite[Definition 6.7]{tbp}}
Let $P$ be a $T$-module. We say that $P$ is projective if the functor Hom$_{T}(P,-):$ $T$\textbf{-mod}$\longrightarrow $\textbf{Ah} preserves epimorphisms. That is to say, if for every surjective $T$-linear map $\pi: M\longrightarrow N$ and every $T$-linear map $f: P\longrightarrow N$ there exists a (not necessarily unique) $T$-linear map $\widetilde{f}:P\longrightarrow M$ such that $\pi\circ\widetilde{f}=f$.\\
Diagrammatically,
\[
\xymatrix @R=24pt{
M \ar@{->>}[r]^-{\pi} & N \\
 & P. \ar@{.>}[ul]^-{\tilde{f}} \ar[u]_-{f}
}
\]
\end{definition}

\begin{proposition} Let $P_{1}$, $P_{2}$ be $T$-modules, \emph{Abs}$(P_{1})$, \emph{Abs}$(P_{2})$ not empty and $P=P_{1}\times P_{2}$. If $P$ is a projective $T$-module, then $P_{i}$ ($i=1, 2$) are projective $T$-modules.
\begin{proof} Let $e_{1}$ $\in$ Abs$(P_{1})$, $e_{2}$ $\in$ Abs$(P_{2})$.\\
Define
\begin{equation*}
\begin{aligned}
\varepsilon_{1}:P_{1}&\longrightarrow P=P_{1}\times P_{2}\\
      p_{1}&\longmapsto (p_{1}, e_{2})\\
\end{aligned}
\end{equation*}

\begin{equation*}
\begin{aligned}
\varepsilon_{2}:P_{2}&\longrightarrow P=P_{1}\times P_{2}\\
      p_{2}&\longmapsto (e_{1}, p_{2})\\
\end{aligned}
\end{equation*}

\begin{equation*}
\begin{aligned}
\pi_{1}:P=P_{1}\times P_{2}&\longrightarrow P_{1}\\
     (p_{1}, p_{2})&\longmapsto p_{1}\\
\end{aligned}
\end{equation*}

\begin{equation*}
\begin{aligned}
\pi_{2}:P=P_{1}\times P_{2}&\longrightarrow P_{2}\\
     (p_{1}, p_{2})&\longmapsto p_{2},\\
\end{aligned}
\end{equation*}
where $p_{1}\in P_{1}$, $p_{2}\in P_{2}$. It is easy to verify that~$\varepsilon_{i}$, $\pi_{i}$ ($i=1, 2$)~are well-defined.\\
\indent~Since for all $x_{1}$, $x_{2}$, $x_{3}$ $\in P_{1}$,
\begin{equation*}
\begin{aligned}
\varepsilon_{1}([x_{1}, x_{2}, x_{3}])&=([x_{1}, x_{2}, x_{3}], e_{2})= ([x_{1}, x_{2}, x_{3}], [e_{2}, e_{2}, e_{2}])\\
&=[(x_{1}, e_{2}), (x_{2}, e_{2}), (x_{3}, e_{2})]\\
&=[\varepsilon_{1}(x_{1}), \varepsilon_{1}(x_{2}), \varepsilon_{1}(x_{3})]\\
\end{aligned}
\end{equation*}
and for all $t\in T$, $x_{4}\in P_{1}$,
\begin{equation*}
\begin{aligned}
\varepsilon_{1}(tx_{4}) = (tx_{4}, e_{2}) = (tx_{4}, te_{2}) = t(x_{4}, e_{2}) = t\varepsilon_{1}(x_{4}).
\end{aligned}
\end{equation*}
So~$\varepsilon_{1}$~is a homomorphism of $T$-modules. Analogously, $\varepsilon_{2}$ and $\pi_{i}$~$(i=1, 2)$ are homomorphisms of $T$-modules. Obviously, $\varepsilon_{i}$ ($i=1, 2$) are monic and $\pi_{i}$ ($i=1, 2$)~are epic.\\
\indent Assume that $M$ and $N$ are $T$-modules, $f$ : $M\longrightarrow N$ is a epimorphism of $T$-modules, and $g_{i}$ : $P_{i}\longrightarrow N$ ($i=1, 2$) are homomorphisms of $T$-modules. Since $P$ is a projective $T$-module, there exists a morphism of $T$-modules $h$: $P=P_{1}\times P_{2}\longrightarrow M$ such that the following diagram commutes
$$
\xymatrix @R=24pt{
&
      P=P_{1}\times P_{2} \ar@<+0,8ex>[d]^{\pi_{i}} \ar@{-->}[ddl]_{h} \\&
 P_{i} \ar[u]^{\varepsilon_{i}} \ar[d]^{g_{i}} \ar@{-->}[dl]_{\hat{h}} \\
M \ar[r]^f & N.\\
}$$
That is, $fh = g_{i}\pi_{i}$ ($i=1, 2$). Let $\hat{h}$ $= h\varepsilon_{i}$ : $P_{i}\longrightarrow M$, then
$f\hat{h} = fh\varepsilon_{i} = g_{i}\pi_{i}\varepsilon_{i} = g_{i}~(i=1, 2)$. So $P_{i}$ ($i=1, 2$) are projective $T$-modules.\\
\end{proof}
\end{proposition}

\begin{proposition}{\rm \cite[Proposition 6.5]{tbp}}\label{ae} Let $\phi$ : $M\longrightarrow N$ and $\psi$ : $N\longrightarrow P$ be morphisms of $T$-modules. Assume that $\phi$ is injective, that $\psi$ admits a section $\delta$ (in particular, it is surjective) and that $\xymatrix{
M \ar[r]^\phi & N \ar[r]^\psi & P
}$ is exact. Then there exists $e^{\prime}\in M$ yielding an isomorphism of $T$-modules $N\cong M^{(e^{\prime})}\times P$, where $M^{(e^{\prime})}$ denotes the $e^{\prime}$-induced left $T$-module structure on $M$.
\end{proposition}

\begin{theorem}(The Schanuel Lemma on Projective $T$-modules) Let $T$ be a truss. Suppose that the following two sequences of $T$-modules are exact. \\
$$
\xymatrix @R=24pt{
\star \ar@{.>}[r] & K \ar@<+0.3ex>@{.>}[r]^{i} \ar@<-0.3ex>[r] & P \ar[r]^\pi & M \ar[r] & \star}
$$
$$
\xymatrix @R=24pt{
\star \ar@{.>}[r] & K^{\prime} \ar@<+0.3ex>@{.>}[r]^{i^{\prime}} \ar@<-0.3ex>[r] & P^{\prime} \ar[r]^{\pi^{\prime}} & M \ar[r] & \star}
.$$\\
 Namely, there exist $m\in M$, $m^{\prime}\in M$ such that \emph{Im}$i$=\emph{ker}$_{m}\pi$, \emph{Im}$i^{\prime}$=\emph{ker}$_{m^{\prime}}\pi^{\prime}$. Moreover,~\emph{Abs}$(K^{\prime})$ is not empty, and $P$, $P^{\prime}$ are projective $T$-modules. If $m = m^{\prime}$, then there exists $e\in K$ yielding an isomorphism of $T$-modules $K^{(e)}\times P^{\prime}\cong K^{\prime}\times P$, where  $K^{(e)}$ denotes the $e$-induced $T$-module structure on $K$.
\begin{proof} Considering the following commutative diagram:\\
$$
\xymatrix @R=24pt{
\star \ar@{.>}[r] & K \ar@<+0.3ex>@{.>}[r]^{i} \ar@<-0.3ex>[r] \ar@{-->}[d]^{\alpha}& P \ar[r]^{\pi} \ar@{-->}[d]^{\beta} & M \ar@{>}[r]^{} \ar[d]^{1_{M}}\ar@{>}[r]^{} & \star & \\
\star \ar@{.>}[r] & K^{\prime} \ar@<+0.3ex>@ {.>}[r]^{i^{\prime}} \ar@<-0.3ex>[r] & P^{\prime} \ar[r]^{\pi^{\prime}} & M \ar@{>}[r]^{} & \star}
$$\\
\indent Since $P$ is a projective $T$-module, there exists a morphism of $T$-modules $\beta$: $P\longrightarrow P^{\prime}$ such that $\pi^{\prime}\beta = \pi$. By diagram chasing, there exists a morphism of $T$-modules $\alpha$: $K\longrightarrow K^{\prime}$ such that $i^{\prime}\alpha = \beta i$.\\
Define
\begin{equation*}
\begin{aligned}
\theta:K&\longrightarrow P\times K^{\prime}\\
       k&\longmapsto (i(k), \alpha(k))\\
\end{aligned}
\end{equation*}
and
\begin{equation*}
\begin{aligned}
\psi:P\times K^{\prime}&\longrightarrow P^{\prime}\\
       (p,k^{\prime})&\longmapsto [\beta(p), i^{\prime}(k^{\prime}), i^{\prime}(e^{\prime})],\\
\end{aligned}
\end{equation*}
where $k\in K$, $p\in P$, $k^{\prime}\in K^{\prime}$ and $e\in$ Abs$(K^{\prime})$. It is easy to verify that $\theta$ and $\psi$ are well-defined. \\
\indent First of all, since\\
\begin{equation*}
\begin{aligned}
\theta([k_{1}, k_{2}, k_{3}])&=(i[k_{1}, k_{2}, k_{3}], \alpha[k_{1}, k_{2}, k_{3}])\\
&=([i(k_{1}), i(k_{2}), i(k_{3})], [\alpha(k_{1}), \alpha(k_{2}), \alpha(k_{3})])\\
&=[(i(k_{1}), \alpha(k_{1})), (i(k_{2}), \alpha(k_{2})), (i(k_{3}), \alpha(k_{3}))]\\
&=[\theta(k_{1}), \theta(k_{2}), \theta(k_{3})],
\end{aligned}
\end{equation*}
\begin{equation*}
\begin{aligned}
\theta(tk_{4})&=(i(tk_{4}), \alpha(tk_{4}))=(ti(k_{4}), t\alpha(k_{4}))=t(i(k_{4}), \alpha(k_{4}))=t\theta((k_{4})
\end{aligned}
\end{equation*}
for all $k_{1}, k_{2}, k_{3}, k_{4}\in K, t\in T$. And for all $(p_{1}, k_{1}^\prime), (p_{2}, k_{2}^\prime), (p_{3}, k_{3}^\prime), (p_{4}, k_{4}^\prime)\in P\times K^{\prime}, t\in T$,
\begin{equation*}
\begin{aligned}
\psi([(p_{1}, k_{1}^\prime), (p_{2}, k_{2}^\prime), (p_{3}, k_{3}^\prime)])&=\psi([p_{1}, p_{2}, p_{3}], [k_{1}^\prime, k_{2}^\prime, k_{3}^\prime])\\
&=[\beta[p_{1}, p_{2}, p_{3}], i^\prime[k_{1}^\prime, k_{2}^\prime, k_{3}^\prime], i^\prime(e^\prime)]\\
&=[\beta[p_{1}, p_{2}, p_{3}], i^\prime[k_{1}^\prime, k_{2}^\prime, k_{3}^\prime], [i^\prime(e^\prime), i^\prime(e^\prime), i^\prime(e^\prime)]]\\
&=[[\beta(p_{1}), \beta(p_{2}), \beta(p_{3})], [i^\prime(k_{1}^\prime), i^\prime(k_{2}^\prime), i^\prime(k_{3}^\prime)], [i^\prime(e^\prime), i^\prime(e^\prime), i^\prime(e^\prime)]]\\
\overset{ Lemma~\ref{le.se}}{=}&[[\beta(p_{1}), i^\prime(k_{1}^\prime), i^\prime(e^\prime)], [\beta(p_{2}), i^\prime(k_{2}^\prime), i^\prime(e^\prime)], [\beta(p_{3}), i^\prime(k_{3}^\prime), i^\prime(e^\prime)]]\\
&=[\psi(p_{1}, k_{1}^\prime), \psi(p_{2}, k_{2}^\prime), \psi(p_{3}, k_{3}^\prime)],
\end{aligned}
\end{equation*}
\begin{equation*}
\begin{aligned}
\psi(t(p_{4}, k_{4}^\prime))&=\psi(tp_{4}, tk_{4}^\prime)\\
&=[\beta(tp_{4}), i^\prime(tk_{4}^\prime), i^\prime(e^\prime)]\\
&= [\beta(tp_{4}), i^\prime(tk_{4}^\prime), i^\prime(te^\prime)]\\
&=[t\beta(p_{4}), ti^\prime(k_{4}^\prime), ti^\prime(e^\prime)]\\
&=t[\beta(p_{4}), i^\prime(k_{4}^\prime), i^\prime(e^\prime)]\\
&=t\psi(p_{4}, k_{4}^\prime).\\
\end{aligned}
\end{equation*}
Then $\theta$ and $\psi$ are homomorphisms of $T$-modules.\\
\indent Secondly, to prove that $\theta$ is monic. Assume that $x\in$ ker$_{\theta(k)}\theta$, obviously, $k\in$ ker$_{\theta(k)}\theta$, then $\theta(x) = \theta(k)$. Hence $(i(x), \alpha(x)) = (i(k), \alpha(k))$, this implies that $i(x) = i(k)$. Since $i$ is monic, $x = k$. Thus, ker$_{\theta(k)}\theta$ is a singleton. By Corollary \ref{ua}, $\theta$ is monic.\\
\indent Thirdly, it suffices to show that $\psi$ is epic. If $p^\prime\in P$, then $\pi^{\prime}(p^{\prime}) = m$, where $m\in M$. Since $\pi$ is epic, there exists $p\in P$ such that $\pi(p) = m$. Thus, $\pi^\prime(p^\prime) = \pi(p)$, this means that $\pi^\prime(p^\prime)=\pi(p)\overset{\pi^{\prime}\beta=\pi}{=} \pi^\prime\beta(p)$. So
\begin{equation*}
\begin{aligned}
\pi^\prime([\beta(p), p^\prime, i^\prime(e^\prime)])=[\pi^\prime\beta(p), \pi^\prime (p^{\prime}), \pi^\prime i^\prime(e^\prime)]
\overset{Lemma~\ref{le.se}}{=}&\pi^\prime i^\prime(e^\prime)=m^\prime,
\end{aligned}
\end{equation*}
then $[\beta(p), p^\prime, i^\prime(e^\prime)]\in$ ker$_{m^\prime}\pi^\prime$ = Im$i^\prime$. Since $i^\prime$ is monic, there exists a unique $k^\prime\in K^\prime$ such that $i^\prime(k^\prime)=[\beta(p), p^\prime, i^\prime(e^\prime)]$, and hence by Lemma~\ref{le.se}, $p^\prime=[ i^\prime(e^\prime), i^\prime(k^\prime), \beta(p)] = [\beta(p), i^\prime(k^\prime), i^\prime(e^\prime)]$, then $\psi(p, k^\prime) = [\beta(p), i^\prime(k^\prime), i^\prime(e^\prime)] = p^\prime$. Thus, proving that $\psi$ is epic.\\
\indent Fourthly, it suffices to show that
\[\xymatrix @R=24pt{
\star \ar@{.>}[r] & K \ar@<+0.3ex>@{.>}[r]^-{\theta} \ar@<-0.3ex>[r] & P\times K^\prime \ar[r]^-\psi & P^\prime \ar[r] & \star}
\]
is exact. Since
$$\psi\theta(k) = \psi(i(k), \alpha(k))\overset{\beta i = i^\prime\alpha}{=}[\beta i(k), i^\prime\alpha(k), i^\prime(e^\prime)]\overset{ Lemma~\ref{le.se}}{=} i^\prime(e^\prime)$$ for any $k\in K$, this means that Im$\theta\subseteq$ ker$_{i^\prime(e^\prime)}\psi$. Let $(p, k^\prime)\in$ ker$_{i^\prime(e^\prime)}\psi$. So $\psi(p, k^\prime) = [\beta(p), i^\prime(k^\prime), i^\prime(e^\prime)]$= $i^\prime(e^\prime)$, then by Lemma~\ref{le.se}, $\beta(p)=i^\prime(k^\prime)$. Since $\pi^\prime\beta(p)\overset{\pi^\prime\beta=\pi}{=}\pi(p) = \pi^\prime i^\prime(k^\prime)=m^\prime=m$, and hence $\pi(p)=m$, this implies that $p\in$ ker$_{m}\pi$ = Im$i$. Since $i$ is monic, there exists a unique $k\in K$ such that $i(k) = p$.\\
\indent~Additional, $i^\prime\alpha(k)\overset{i^\prime\alpha = \beta i}{=} \beta i(k) = \beta(p) = i^\prime(k^\prime)$. Since $i^\prime$ is monic, $\alpha(k) = k^\prime$. Thus, $\theta(k) = (i(k), \alpha(k)) = (p, k^\prime)$, this means that Im$\theta\supseteq$ ker$_{i^\prime(e^\prime)}\psi$. This shows that Im$\theta$ = ker$_{i^\prime(e^\prime)}\psi$. Thus, proving that the sequence is exact.\\
\indent Finally, to prove that $\xymatrix @R=24pt{
\star \ar@{.>}[r] & K \ar@<+0.3ex>@{.>}[r]^-{\theta} \ar@<-0.3ex>[r] & P\times K^\prime \ar[r]^-\psi & P^\prime \ar[r] & \star}$~splits. Since $P^\prime$ is a projective $T$-module, there exists a morphism of $T$-modules $\gamma$: $P^\prime\longrightarrow P\times K^\prime$ such that $\psi\gamma= 1_{P^\prime}$.
$$
\xymatrix @R=24pt{
\star \ar@{.>}[r] & K \ar@<+0.3ex>@{.>}[r]^-{\theta} \ar@<-0.3ex>[r] & P\times K^\prime \ar[r]^-{\psi} & P^\prime \ar@/^3ex/[l]^-{\gamma} \ar[r] & \star
}
$$
Thus,~$\xymatrix @R=24pt{
\star \ar@{.>}[r] & K \ar@<+0.3ex>@{.>}[r]^-{\theta} \ar@<-0.3ex>[r] & P\times K^\prime \ar[r]^-\psi & P^\prime \ar[r] & \star}$~splits. Then by Proposition \ref{ae}, there exists $e\in K$ yielding an isomorphism of $T$-modules $K^{(e)}\times P^{\prime}\cong K^{\prime}\times P$, where  $K^{(e)}$ denotes the $e$-induced $T$-module structure on $K$.
\end{proof}
\end{theorem}

\section{Injective modules over trusses}

As the dual notion of projective modules over trusses, the concept of injective modules over trusses is introduced and its properties are discussed in detail. The Schanuel Lemma on injective modules over trusses is posed in this section.

\begin{definition}
Let $E$ be a $T$-module. If for every monomorphism of $T$-modules $i$ : $M\longrightarrow N$ and every homomorphism of $T$-modules $f$ : $M\longrightarrow E$, there exists a morphism of $T$-modules $g$ : $N\longrightarrow E$ such that $f=gi$, that is to say, the following diagram is commutative:
$$\xymatrix @R=24pt{M \ar[r]^-{i} \ar[d]_-f& N \ar@{-->}[dl]^-{{g}}
\\
 E,}$$
then we say $E$ is an injective $T$-module.
\end{definition}

\begin{example}  The singleton $T$-module is injective.
\begin{proof} Assume that $M$ and $N$ are $T$-modules, $i$ : $M\longrightarrow N$ is a monomorphism of $T$-modules, $\alpha$ : $M\longrightarrow\star$ is a morphism of $T$-modules. Let $g$ : $N\longrightarrow \star$ be a morphism of $T$-modules such that $g(n) = \ast$ for every $n\in N$. It is easy to see that $gi = \alpha$. Namely, the following diagram is commutative:
$$\xymatrix @R=24pt{M \ar[r]^-{i} \ar[d]_\alpha & N \ar@{-->}[dl]^-{{g}}
\\
 \star.}
$$
So the singleton $T$-module is injective.
\end{proof}
\end{example}

\begin{proposition}{\rm \cite[Proposition 6.2]{tbp}}\label{we} Let $\phi$ : $M\longrightarrow N$ and $\psi$ : $N\longrightarrow P$ be morphisms of $T$-modules. Assume that $\psi$ is surjective, that $\phi$ admits a retraction $\gamma$ (in particular, it is injective) and that $\xymatrix{
M \ar[r]^\phi & N \ar[r]^\psi & P
}$ is exact. Then $N\cong M\times P$ as $T$-modules. We will call such a sequence a split exact sequence.
\end{proposition}

\begin{proposition} \label{we1} Let $E$ be an injective $T$-module. Then the exact sequence of ~$T$-modules\\
$\xymatrix @R=24pt{
\star \ar@{.>}[r] & E \ar@<+0.3ex>@{.>}[r]^{f} \ar@<-0.3ex>[r] & M \ar[r]^g & N \ar[r] & \star
}$ splits.
\begin{proof} Since $E$ is an injective $T$-module, there exists a morphism of $T$-modules $j$: $M\longrightarrow E$ such that $jf = 1_{E}$.
$$
\xymatrix @R=24pt{
\star \ar@{.>}[r] & E \ar@<+0.5ex>@{.>}[r]^{f} \ar[r] & M \ar[r]^-{g} \ar@/^3ex/[l]^-{j} & N \ar[r] & \star
}
$$
By Proposition~\ref{we},
$\xymatrix @R=24pt{
\star \ar@{.>}[r] & E \ar@<+0.3ex>@{.>}[r]^{f} \ar@<-0.3ex>[r] & M \ar[r]^g & N \ar[r] & \star
}$ splits.
\end{proof}
\end{proposition}

\begin{definition}
Let $M$ be a $T$-module and $E$ a submodule of $M$. If there exists a $T$-module $K$ such that $M\cong E\times K$, then we say $E$ is a direct factor of $M$.
\end{definition}

\begin{corollary} Let $M$ be a $T$-module and $E$ a submodule of $M$. If $E$ is an injective $T$-module, then $E$ is a direct factor of $M$.
\begin{proof} By Proposition \ref{we1},~$\xymatrix @R=24pt{
\star \ar@{.>}[r] & E \ar@<+0.3ex>@{.>}[r]^{f} \ar@<-0.3ex>[r] & M \ar[r]^g & M/E\ar[r] & \star
}$~splits. Thus, by Proposition \ref{we}, $M\cong E\times M/E$ as $T$-modules. Then $E$ is a direct factor of $M$.
\end{proof}
\end{corollary}

\begin{proposition}\label{pe} Let $E_{1}$, $E_{2}$ be $T$-modules, \emph{Abs}$(E_{1})$, \emph{Abs}$(E_{2})$ not empty and $E=E_{1}\times E_{2}$. Then $E$ is an injective $T$-module if and only if $E_{1}$ and $E_{2}$ are injective $T$-modules.
\begin{proof} Let $e_{1}$ $\in$ Abs$(E_{1})$, $e_{2}$ $\in$ Abs$(E_{2})$.\\
Define
\begin{equation*}
\begin{aligned}
\varepsilon_{1}:E_{1}&\longrightarrow E=E_{1}\times E_{2}\\
      m_{1}&\longmapsto (m_{1}, e_{2})\\
\end{aligned}
\end{equation*}

\begin{equation*}
\begin{aligned}
\varepsilon_{2}:E_{2}&\longrightarrow E=E_{1}\times E_{2}\\
      m_{2}&\longmapsto (e_{1}, m_{2})\\
\end{aligned}
\end{equation*}

\begin{equation*}
\begin{aligned}
\pi_{1}:E=E_{1}\times E_{2}&\longrightarrow E_{1}\\
     (m_{1}, m_{2})&\longmapsto m_{1}\\
\end{aligned}
\end{equation*}

\begin{equation*}
\begin{aligned}
\pi_{2}:E=E_{1}\times E_{2}&\longrightarrow E_{2}\\
     (m_{1}, m_{2})&\longmapsto m_{2},\\
\end{aligned}
\end{equation*}
where $m_{1}\in E_{1}$, $m_{2}\in E_{2}$. It is easy to verify that~$\varepsilon_{i}$, $\pi_{i}$ $(i=1, 2)$~are well-defined.\\
\indent Since
\begin{equation*}
\begin{aligned}
\varepsilon_{1}([m_{1}^{\prime}, m_{1}^{\prime\prime}, m_{1}^{\prime\prime \prime}])&=([m_{1}^{\prime}, m_{1}^{\prime\prime}, m_{1}^{\prime\prime \prime}], e_{2})= ([m_{1}^{\prime}, m_{1}^{\prime\prime}, m_{1}^{\prime\prime \prime}], [e_{2}, e_{2}, e_{2}])\\
&=[(m_{1}^{\prime}, e_{2}), (m_{1}^{\prime\prime}, e_{2}), (m_{1}^{\prime\prime \prime}, e_{2})]\\
&=[\varepsilon_{1}(m_{1}^{\prime}), \varepsilon_{1}(m_{1}^{\prime\prime}), \varepsilon_{1}(m_{1}^{\prime\prime\prime})],
\end{aligned}
\end{equation*}

\begin{equation*}
\begin{aligned}
\varepsilon_{1}(tm_{1}^{\prime}) = (tm_{1}^{\prime}, e_{2}) = (tm_{1}^{\prime}, te_{2}) = t(m_{1}^{\prime}, e_{2}) = t\varepsilon_{1}(m_{1}^{\prime})
\end{aligned}
\end{equation*}
for all $m_{1}^{\prime}$, $m_{1}^{\prime \prime}$, $m_{1}^{\prime\prime\prime} \in E_{1}$, $t\in T$, $\varepsilon_{1}$~is a homomorphism of $T$-modules. Analogously, $\varepsilon_{2}$, $\pi_{i}$ $(i=1 ,2)$ are homomorphisms of $T$-modules. Obviously, $\varepsilon_{i}$ $(i=1, 2)$ are monic and $\pi_{i}$ $(i=1, 2)$ are epic.\\
\indent ``$\Rightarrow$" Assume that $M$, $N$ are $T$-modules, $f$ : $M\longrightarrow N$ is a monomorphism of $T$-modules and $g$ : $M\longrightarrow E_{1}$ is a morphism of $T$-modules. Since $E=E_{1}\times E_{2}$ is an injective $T$-module, there exists a morphism of $T$-modules $h$: $N\longrightarrow E=E_{1}\times E_{2}$ such that the following diagram\\
$$\xymatrix @R=24pt{
M\ar[r]^f \ar[d]^g & N \ar@{-->}[dl]_{\gamma} \ar@{-->}[ddl]^h\\
E_1 \ar[d]^{\varepsilon_1}  \\
E=E_1 \times E_2 \ar@<+0.8ex>[u]^{\pi_1}
}
$$
commutes, that is, $hf = \varepsilon_{1}g$. Let $\gamma$ $= \pi_{1}h$ : $N\longrightarrow E_{1}$. Then
$\gamma$$f = \pi_{1}hf = \pi_{1}\varepsilon_{1}g = g$, so $E_{1}$ is an injective $T$-module. Similarly, $E_{2}$ is an injective $T$-module.\\
\indent``$\Leftarrow$'' Assume that $M$, $N$ are $T$-modules, $f$ : $M\longrightarrow N$ is a monomorphism of $T$-modules and $g$ : $M\longrightarrow E=E_{1}\times E_{2}$ is a morphism of $T$-modules. By injectivity of $E_{1}$ and $E_{2}$, there exist morphisms of $T$-modules~$\beta_{i}$ : $N\longrightarrow E_{i}$ ($i=1, 2$) such that the  following diagram\\
$$\xymatrix @R=24pt{
M\ar[r]^f \ar[d]^g & N \ar@{-->}[dl]_{\alpha} \ar@{-->}[ddl]^{\beta_{i}}\\
E=E_1\times E_2 \ar[d]^{\pi_i}  \\
E_i \ar@<+0.8ex>[u]^{\varepsilon_i}
}
$$
commutes, that is, $\beta_{i}f = \pi_{i}g$ ($i=1, 2$). By the universal property of product, there exists a morphism of $T$-modules $\alpha$ : $N\longrightarrow E$ such that $\beta_{i}=\pi_{i}\alpha$, then $\beta_{i}f=\pi_{i}\alpha f=\pi_{i}g$ ($i=1, 2$). Thus, $\pi_{i}g(m)$= $\pi_{i}\alpha f(m)$ for any $m\in M$. Let $g(m)=(s, t)$ and $\alpha f(m)=(x, y)$, then $s=x$ and $t=y$. So $g(m)=\alpha f(m)$, that is to say, $\alpha f = g$. Then~$E$ is an injective $T$-module.
\end{proof}
\end{proposition}

\begin{definition} Let $T$ be a truss with zero element $0$ and $0\neq t\in T$. If there exists $0\neq t^\prime\in T$ such that $tt^\prime =0$, then we say $t$ is a left absorber of $t^\prime$, $t^\prime$ is a right absorber of $t$. If $t$ is both a left and right absorber of $t^\prime$, then  we say $t$ is  an absorber factor of $T$.
\end{definition}

\begin{proposition} A truss $T$ (with zero element $0$) has no absorber factors if and only if the cancellation law holds in $T$.
\begin{proof} ``$\Rightarrow$"  Let $T$ have no absorber factors. If~$tt^\prime = tt^{\prime\prime}$~for any $0\neq t,~t^\prime, t^{\prime\prime}\in T$, then $[tt^{\prime}, tt^{\prime\prime}, 0] = 0$, and hence $0 = [tt^{\prime}, tt^{\prime\prime}, t0] = t[t^{\prime}, t^{\prime\prime}, 0]$.
Since $T$ has no absorber factors and $0\neq t$, $[t^{\prime}, t^{\prime\prime}, 0] = 0$. By Lemma \ref{le.se}, $t^{\prime} = t^{\prime\prime}$, and so the left cancellation law holds. Similarly, the right cancellation law holds.\\
\indent``$\Leftarrow$"  Assume that $T$ has absorber factors, then for some $0\neq t\in T$, there exists $0\neq t^{\prime}\in T$ such that $tt^{\prime} = t^{\prime}t = 0$.
Due to $tt^{\prime} = t^{\prime}t = 0 = t0 = 0t$, by the cancellation law holds, $t^{\prime} = 0$. Contradictory with the hypothesis, so $T$ has no absorber factors.
\end{proof}
\end{proposition}

\begin{definition} If $T$ is an abelian truss (with zero element $0$) which has no absorber factors, then we say $T$ is a domain truss.
\end{definition}

\begin{example}
Every domain may be regarded as a domain truss. For example, T$(\mathbb{Z})$ and T$(\mathbb{Z}[x])$.
\end{example}

\begin{definition} Let $T$ be a domain truss and $M$ a $T$-module.
If for any $0\neq t\in T$ and $m\in M$, there exists $m^{\prime}\in M$ such that $m = tm^{\prime}$, then we say $M$ is a divisible $T$-module.
\end{definition}
\begin{example} (1)~$\mathbb{Q}$ is divisible as a module over the T$(\mathbb{Z})$.\\
\indent(2)~Every singleton $T$-module over a domain truss is divisible.\\
\indent(3)~T$(\mathbb{Z})$ is not divisible as a module over itself.
\end{example}

\begin{proposition} Let $T$ be a domain truss, $M_{1}$ and $M_{2}$ $T$-modules. If $M=M_{1}\times M_{2}$ is a divisible $T$-module, then $M_{1}$ and $M_{2}$ are divisible $T$-modules.
\begin{proof} Since $M=M_{1}\times M_{2}$ is a divisible $T$-module, then for any $0\neq t\in T$ and $m=(m_{1}, m_{2})\in M=M_{1}\times M_{2}$, there exists $m^{\prime}=(m_{1}^{\prime}, m_{2}^{\prime})\in M$ such that $m=(m_{1}, m_{2})=tm^{\prime}=t(m_{1}^{\prime}, m_{2}^{\prime})$, and hence $(m_{1}, m_{2})=(tm_{1}^{\prime}, tm_{2}^{\prime})$. So $m_{1}=tm_{1}^{\prime}$ and $m_{2}=tm_{2}^{\prime}$, imply $M_{1}$ and $M_{2}$ are divisible $T$-modules.
\end{proof}
\end{proposition}

\begin{theorem} Suppose that the sequence of $T$-modules $\xymatrix @R=24pt{ M \ar[r]^{f} & N \ar[r]^g & P }$ is \emph{Abs}-exact and $g$ is epic. If $M$, $P$ are divisible $T$-modules, then $N$ is a divisible $T$-module.
\begin{proof}
For any $0\neq t\in T$ and $n\in N$, since $P$ is a divisible $T$-module, there exists $p^{\prime}\in P$ such that $g(n)=tp^{\prime}$. As $g$ is epic, there exists $n^{\prime}\in N$ such that $g(n^{\prime})=p^{\prime}$, and hence $g(n)=tg(n^{\prime})$.\\
\indent However, since $\xymatrix @R-24pt{ M \ar[r]^{f} & N \ar[r]^g & P }$ is \emph{Abs}-exact,
$$g[n, tn^{\prime}, tf(m)]=[g(n), g(tn^{\prime}), g(tf(m))]=[g(n), tg(n^{\prime}), t(gf(m))] = t(gf(m)) = te = e,$$ where $f(m)\in$ Im$f$, and hence $[n, tn^{\prime}, tf(m)]\in$ ker$_{e}g$ = Im$f$. Thus, there exists $m^{\prime}\in M$ such that $f(m^{\prime}) = [n, tn^{\prime}, tf(m)]$, and hence $n = [f(m^{\prime}), tf(m), tn^{\prime}]$ by Lemma~\ref{le.se}.\\
\indent Since $M$ is a divisible $T$-module, there exists $m^{\prime\prime}\in M$ such that $m^{\prime}=tm^{\prime\prime}$. So
$$n=[f(m^{\prime}), tf(m), tn^{\prime}]=[f(tm^{\prime\prime}),tf(m),tn^{\prime}]=[tf(m^{\prime\prime}),tf(m),tn^{\prime}]=t[f(m^{\prime\prime}),f(m),n^{\prime}].$$
Therefore $N$ is a divisible $T$-module.
\end{proof}
\end{theorem}

\begin{proposition}
Let $M$ be a $T$-module and $N$ a submodule of M. If $M$ is a divisible $T$-module, then $M/N$ is a divisible $T$-module.
\begin{proof}
Let $\overline{m}\in M/N$. For any $0\neq t\in T$, since $M$ is a divisible $T$-module, there exists $m^{\prime}\in M$ such that $m=tm^{\prime}$, and  hence $\overline{m}=\overline{tm^{\prime}}=t\overline{m^{\prime}}$. Then $M/N$ is a divisible $T$-module.
\end{proof}
\end{proposition}

\begin{proposition} Let $T$ be a domain truss with identity $1$ and zero element $0$. If $T$ is divisible as a module over itself, then for any $0\neq t\in T$, there exists $t^{\prime}\in T$ such that $t^{\prime}t=1$.
\begin{proof} Since $T$ is a divisible $T$-module, then for any $0\neq t\in T$, there exists $t^{\prime}\in T$ such that $1 = tt^{\prime}$. However, $t = 1t = tt^{\prime}t$, this means that $[t, tt^{\prime}t, 0] = 0$. Thus, $[t, tt^{\prime}t, 0] =[t, tt^{\prime}t, t0] = t[1, t^{\prime}t, 0]= 0$. Since $0\neq t\in T$ and $T$ is a domain truss, $[1, t^{\prime}t, 0] = 0$, and hence $t^{\prime}t = 1$ by Lemma \ref{le.se}.
\end{proof}
\end{proposition}

\begin{theorem} Let $T$ be a domain truss with identity $1$ and $M$ a normalised $T$-module. If $M$ is an injective $T$-module, then $M$ is a divisible $T$-module.
\begin{proof}
Let $M$ be an injective $T$-module. Taking $m\in M$, $0\neq t\in T$. Defined~$f_{t}$ : $T\longrightarrow T$ by~$f_{t}(t^{\prime}) = tt^{\prime}$~for every $t^{\prime}\in T$. It is easy to check that $f_{t}$ is a map.\\
\indent For all $t_{1}, t_{2}, t_{3}\in T$,
\begin{equation*}
\begin{aligned}
f_{t}([t_{1}, t_{2}, t_{3}]) = t[t_{1}, t_{2}, t_{3}] = [tt_{1}, tt_{2}, tt_{3}] = [f_{t}(t_{1}), f_{t}(t_{2}), f_{t}(t_{3})]
\end{aligned}
\end{equation*}
and for all $t_{4}, t_{5}\in T$,
$$f_{t}(t_{4}t_{5}) = t(t_{4}t_{5}) = (tt_{4})t_{5} = (t_{4}t)t_{5} = t_{4}(tt_{5}) = t_{4}f_{t}(t_{5}).$$
Thus,~$f_{t}$ is a homomorphism of $T$-modules.\\
\indent Let~$a$, $b\in$ ker$_{x}f_{t}$, where $x\in$ Im$f_{t}$. Since~$f_{t}(a) = x =f_{t}(b)$, $ta = tb$, and hence $0 = [ta, tb, 0] = [ta, tb, t0] = t[a, b, 0]$. Since $T$ is a domain truss and $0\neq t\in T$, $[a, b, 0] = 0$. By Lemma \ref{le.se}, $a = b$, that is to say, ker$_{x}f_{t}$ is a singleton. By Corollary \ref{ua}, $f_{t}$ is monic.\\
\indent Define
\begin{equation*}
\begin{aligned}
g_{m}: T &\longrightarrow M\\
       x&\longmapsto xm,\\
\end{aligned}
\end{equation*}
where $x\in T$. It is easy to see that $g_{m}$ is well-defined.\\
\indent For all $t^{\prime}_{1}, t^{\prime}_{2}, t^{\prime}_{3}\in T$,
\begin{equation*}
\begin{aligned}
g_{m}([t^{\prime}_{1}, t^{\prime}_{2}, t^{\prime}_{3}]) = [t^{\prime}_{1}, t^{\prime}_{2}, t^{\prime}_{3}]m = [t^{\prime}_{1}m, t^{\prime}_{2}m, t^{\prime}_{3}m] = [g_{m}(t^{\prime}_{1}), g_{m}(t^{\prime}_{2}), g_{m}(t^{\prime}_{3})]
\end{aligned}
\end{equation*}
and for all $t^{\prime}_{4}$, $t^{\prime}_{5}\in T$,
\begin{equation*}
\begin{aligned}
g_{m}(t^{\prime}_{4}t^{\prime}_{5}) = (t^{\prime}_{4}t^{\prime}_{5})m = t^{\prime}_{4}(t^{\prime}_{5}m) = t^{\prime}_{4}g_{m}(t^{\prime}_{5}).
\end{aligned}
\end{equation*}
So $g_{m}$ is a homomorphism of $T$-modules.\\
\indent Since $M$ is an injective $T$-module, there exists a morphism of $T$-modules $h$ : $T\longrightarrow M$ such that $g_{m} = hf_{t}$, that is to say, the following diagram commutes\\
\[
\xymatrix @R=24pt{ T \ar[r]^-{f_{t}} \ar[d]_{g_{m}} & T \ar@{-->}[dl]^-{{h}}
\\
 M. }
\]
\indent Since $T$ is a domain truss with identity 1 and $M$ a normalised $T$-module, $m = 1_{T}m = g_{m}(1_{T}) = hf_{t}(1_{T}) = h(t1_{T})= th(1_{T})$. Thus, $M$ is a divisible $T$-module.\\
\end{proof}
\end{theorem}

\begin{example}
$\mathbb{Q}$ is both an injective T($\mathbb{Z})$-module and  a divisible T($\mathbb{Z}$)-module.
\end{example}

\begin{remark} \label{qe}
Let $M$ and $N$ be $T$-modules, Then the set ${\rm Hom}_{T}(M,N)$ is an abelian heap with the point wise heap operation.
\end{remark}

\begin{proposition} \label{fe}Let $S$, $T$ be trusses and $f: S\longrightarrow T$ a homomorphism of trusses. Then\\
\indent {\rm (1)}~$T$ is a left $S$-module;\\
\indent {\rm (2)}~If $M$ is a left $S$-module, then \emph{Hom}$_{S}(T, M)$ is a left $T$-module.
\begin{proof} (1) Define
\begin{equation*}
\begin{aligned}
S\times T&\longrightarrow T\\
       (s, t)&\longmapsto s\cdot t = f(s)t\\
\end{aligned}
\end{equation*}
Since\\
\begin{equation*}
\begin{aligned}
  \left[s_{1},s_{2},s_{3}\right]\cdot t&= f([s_{1}, s_{2}, s_{3}])t\\
  &=[f(s_{1}), f(s_{2}), f(s_{3})]t\\
  &=[f(s_{1})t, f(s_{2})t, f(s_{3})t]\\
  &=[s_{1}\cdot t, s_{2}\cdot t, s_{3}\cdot t],
\end{aligned}
\end{equation*}

\begin{equation*}
\begin{aligned}
s_{1}\cdot[t, t^{\prime}, t^{\prime\prime}]&= f(s_{1})[t, t^{\prime}, t^{\prime\prime}]\\
&= [f(s_{1})t, f(s_{1})t^{\prime}, f(s_{1})t^{\prime\prime}]\\
&=[s_{1}\cdot t, s_{1}\cdot t^{\prime}, s_{1}\cdot t^{\prime\prime}],
\end{aligned}
\end{equation*}

\begin{equation*}
\begin{aligned}
s_{1}\cdot(s_{2}\cdot t)&= s_{1}\cdot(f(s_{2})t) = f(s_{1})(f(s_{2})t) \\
&= (f(s_{1})f(s_{2}))t = f(s_{1}s_{2})t = (s_{1}\cdot s_{2})\cdot t
\end{aligned}
\end{equation*}
for all $s_{1}, s_{2}, s_{3} \in S$ and $t, t^{\prime}, t^{\prime\prime}\in T$, $T$ is a left $S$-module.\\

(2) By Remark~\ref{qe}, Hom$_{S}(T, M)$ is an abelian heap.\\
\indent~Define\\
\begin{equation*}
\begin{aligned}
T\times \mathrm{Hom}_{S}(T, M) &\longrightarrow \mathrm{Hom}_{S}(T, M)\\
  (t, g)&\longmapsto (t\cdot g):t^{\prime}\mapsto g(t^{\prime}t)\\
\end{aligned}
\end{equation*}
It is easy to see that $t\cdot g$ is a homomorphism of $T$-modules. \\
Since\\
\begin{equation*}
\begin{aligned}
([t_{1}, t_{2}, t_{3}]\cdot g_{1})(t^{\prime})&= g_{1}(t^{\prime}[t_{1}, t_{2}, t_{3}])=g_{1}([t^{\prime}t_{1}, t^{\prime}t_{2}, t^{\prime}t_{3}])\\
&= [g_{1}(t^{\prime}t_{1}), g_{1}(t^{\prime}t_{2}), g_{1}(t^{\prime}t_{3})]\\
&=[(t_{1}\cdot g_{1})(t^{\prime}), (t_{2}\cdot g_{1})(t^{\prime}), (t_{3}\cdot g_{1})(t^{\prime})]\\
&= [t_{1}\cdot g_{1}, t_{2}\cdot g_{1}, t_{3}\cdot g_{1}](t^{\prime}),
\end{aligned}
\end{equation*}

\begin{equation*}
\begin{aligned}
(t_{1}\cdot[g_{1}, g_{2}, g_{3}])(t^{\prime})&= [g_{1}, g_{2}, g_{3}](t^{\prime}t_{1})=[g_{1}(t^{\prime}t_{1}), g_{2}(t^{\prime}t_{1}), g_{3}(t^{\prime}t_{1})]\\
&=[(t_{1}\cdot g_{1})(t^{\prime}), (t_{1}\cdot g_{2})(t^{\prime}), (t_{1}\cdot g_{3})(t^{\prime})]\\
&= [t_{1}\cdot g_{1}, t_{1}\cdot g_{2}, t_{1}\cdot g_{3}](t^{\prime}),
\end{aligned}
\end{equation*}

\begin{equation*}
\begin{aligned}
((t_{1}\cdot t_{2})\cdot g_{1})(t^{\prime})&=g_{1}(t^{\prime}(t_{1}\cdot t_{2})) = g_{1}((t^{\prime}t_{1})\cdot t_{2})\\
&= (t_{2}\cdot g_{1})(t^{\prime}t_{1}) = (t_{1}\cdot(t_{2}\cdot g_{1}))(t^{\prime})
\end{aligned}
\end{equation*}
for all $t_{1}, t_{2}, t_{3}, t^{\prime}\in T$ and $g_{1}, g_{2}, g_{3}\in$ Hom$_{S}(T, M)$, Hom$_{S}(T, M)$ is a left $T$-module.
\end{proof}
\end{proposition}

\begin{proposition}\label{ge} A $T$-module $E$ is injective if only and if for every $T$-module $N$, every submodule $N^{\prime}$ of $N$, and for any homomorphism of ~$T$-modules $\varphi: N^{\prime}\longrightarrow E$, there exists a homomorphism of ~$T$-modules $\theta: N\longrightarrow E$ such that the following diagram\\
\[
\xymatrix @R=24pt{ N^{\prime} \ar[r]^-{i} \ar[d]_{\varphi} & N \ar@{-->}[dl]^-{{\theta}}
\\
 E }
\]
commutes, where $i$ is the inclusion map.
\begin{proof}``$\Rightarrow$" Clearly.\\
\indent``$\Leftarrow$" Let $M$ be a $T$-module and $f:M\longrightarrow N$ monic. So $h=f:M\longrightarrow \mathrm {Im}f$ is an isomorphism of $T$-modules. Write $\delta=gh^{-1}:\mathrm{Im}f\longrightarrow E$, where $h^{-1}$ is the inverse of $h$.
By hypothesis, there exists a homomorphism of $T$-modules $\theta: N\longrightarrow E$ such that $\theta i=gh^{-1}=\delta$. Thus, $\theta ih = gh^{-1}h=\delta h= g$, that is to say, the following diagram
$$\xymatrix @R=50pt{
M\ar[r]^-{h} \ar[d]_{g}& \mathrm{Im}f \ar@<+0.5ex>@{>}[l]^-{h^{-1}} \ar@{-->}[dl]_{\delta} \ar[r]^{i}& N \ar@{-->}[dll]^-{\theta}\\
E
}
$$
commutes, as desired.
\end{proof}
\end{proposition}

\begin{proposition} Let $f : S\longrightarrow T$  be a homomorphism of trusses, where $T$ is an abelian truss with identity $1$. If $M$ is an injective $S$-module, then \emph{Hom}$_{S}(T, M)$ is an injective $T$-module.
\end{proposition}
\begin{proof} By Proposition~\ref{fe}, Hom$_{S}(T, M)$ is a left $T$-module. Let $N$ be a left $T$-module and $N^{\prime}$ a submodule of $N$. Assume that $h : N^{\prime}\longrightarrow \mathrm{Hom}_{S}(T, M)$ is a homomorphism of $T$-modules. Note that $N$ is an $S$-module, with scalar multiplication defined by $s\cdot n = f(s)n$ for all $s\in S$ and $n\in N$. Thus, $N^{\prime}$ is an $S$-submodule of $N$.\\
\indent Define
\begin{equation*}
\begin{aligned}
\varphi: N^{\prime} &\longrightarrow N\\
       n^{\prime}&\longmapsto h(n^{\prime})(1_{T}),\\
\end{aligned}
\end{equation*}
where $n^{\prime}\in N^{\prime}$. It is easy to see that $\varphi$ is well-defined.\\
Since\\
\begin{equation*}
\begin{aligned}
\varphi([n_{1}^{\prime}, n_{2}^{\prime}, n_{3}^{\prime}])&=h([n_{1}^{\prime}, n_{2}^{\prime}, n_{3}^{\prime}])(1_{T})\\
&=[h(n_{1}^{\prime}), h(n_{2}^{\prime}), h(n_{3}^{\prime})](1_{T})\\
&=[h(n_{1}^{\prime})(1_{T}), h(n_{2}^{\prime})(1_{T}), h(n_{3}^{\prime})(1_{T})]\\
&=[\varphi(n_{1}^{\prime}), \varphi(n_{2}^{\prime}), \varphi(n_{3}^{\prime})],
\end{aligned}
\end{equation*}
\begin{equation*}
\begin{aligned} \varphi(s\cdot n_{1}^{\prime})&=h(s\cdot n_{1}^{\prime})(1_{T})=h(f(s)n_{1}^{\prime})(1_{T})\\
&=(f(s)h(n_{1}^{\prime}))(1_{T})\\
&=h(n^{\prime})(f(s)1_{T})=h(n^{\prime})(s\cdot1_{T})\\
&=s\cdot h(n^{\prime})(1_{T})=s\cdot \varphi(n^{\prime})
\end{aligned}
\end{equation*}
for all $s\in S$, $n_{1}^{\prime}, n_{2}^{\prime}, n_{3}^{\prime}\in N^{\prime}$, $\varphi$ is a homomorphism of $S$-modules.\\
\indent Since $M$ is an injective $S$-module, there exists a homomorphism of $S$-modules $\theta: N\longrightarrow M$ such that $\theta i=\varphi$ by Proposition~\ref{ge}, that is to say, the following diagram\\
\[
\xymatrix @R=24pt{ N^{\prime} \ar[r]^-{i} \ar[d]_{\varphi} & N \ar@{-->}[dl]^-{{\theta}}
\\
 M }
\]
commutes.\\
\indent Define
\begin{equation*}
\begin{aligned}
\beta: N &\longrightarrow \mathrm {Hom}_{S}(T, M)\\
       n &\longmapsto \beta(n):t \longmapsto \theta(t\cdot n),\\
\end{aligned}
\end{equation*}
where $n\in N, t\in T$. It is easy to see that $\beta$ is a map. \\
Since\\
\begin{equation*}
\begin{aligned}
\beta([n_{1}, n_{2}, n_{3}])(t)&=\theta(t\cdot[n_{1}, n_{2}, n_{3}])\\
&=\theta([t\cdot n_{1}, t\cdot n_{2}, t\cdot n_{3}])\\
&=[\theta(t\cdot n_{1}), \theta(t\cdot n_{2}), \theta(t\cdot n_{3})]\\
&=[\beta(n_{1})(t), \beta(n_{2})(t), \beta(n_{3})(t)]\\
&=[\beta(n_{1}), \beta(n_{2}), \beta(n_{3})](t),
\end{aligned}
\end{equation*}
\begin{equation*}
\begin{aligned}
(\beta(t^{\prime}\cdot n_{1}))(t)&= \theta(t\cdot(t^{\prime}\cdot n_{1}))\\
&=\theta((tt^{\prime})\cdot n_{1}) = \theta ((t^{\prime}t)\cdot n_{1})\\
&=\beta(n_{1})(t^{\prime}t) = (t^{\prime}\cdot\beta(n_{1}))(t)
\end{aligned}
\end{equation*}
for all $t^{\prime}, t\in T$, $n_{1}, n_{2}, n_{3}\in N$, $\beta$ is a homomorphism of $T$-modules. \\
\indent Since for any $x^{\prime}\in N^{\prime}$ and $t\in T$
\begin{equation*}
\begin{aligned}
\beta(i(x^{\prime}))(t)&= \theta(t\cdot i(x^{\prime})) = \theta(i(t\cdot x^{\prime}))=\varphi(t\cdot x^{\prime})\\
&= h(t\cdot x^{\prime})(1_{T})=t\cdot h(x^{\prime})(1_{T})\\
&=h(x^{\prime})(1_{T}t)=h(x^{\prime})(t).
\end{aligned}
\end{equation*}
That is to say, the following diagram\\
\[
\xymatrix @R=24pt{ N^{\prime} \ar[r]^-{i} \ar[d]_{h} & N \ar@{-->}[dl]^-{{\beta}}
\\
 \mathrm {Hom}_{S}(T, M) }
\]
commutes. By Proposition~\ref{ge}, Hom$_{S}(T, M)$ is an injective $T$-module
\end{proof}

\begin{lemma} {\rm \cite[Lemma 4.11]{tb1}}\label{he}~Let $(T, [-, -, -], \cdot)$ be a truss and let $(M, [-, -, -], \alpha_{M})$ be a left $T$-module. For any set $X$, the heap $M^{X}$ of functions from $X$ to $M$ is a module with a pointwise defined action, for all $t\in T$, $x\in X$ and $f\in M^{X}$,
$$(t\triangleright f)(x)=t\triangleright f(x),$$
i.e. $\alpha_{M^{X}}(t)=\emph{Map}(X, \alpha_{M}(t))$.
\end{lemma}

\begin{proposition} Let $T$ be a truss and $X$ a non-empty set. If $E$ is an injective $T$-module, then $E^{X}$ is an injective $T$-module.
\begin{proof} By Lemma~\ref{he}, $E^{X}$ is a left $T$-module. Let $M$ be a left $T$-module and $N$ a submodule of $M$. Assume that $\alpha:N\longrightarrow E^{X}$ is a homomorphism of $T$-modules. \\
\indent For each $x\in X$, define
\begin{equation*}
\begin{aligned}
 \alpha_{x}: N &\longrightarrow E\\
       n &\longmapsto \alpha_{x}(n)=\alpha(n)(x),\\
\end{aligned}
\end{equation*}
where $n\in N$. It is easy to see that $\alpha_{x}$ is well-defined.\\
Since\\
\begin{equation*}
\begin{aligned}
\alpha_{x}([n_{1},n_{2},n_{3}])&=\alpha([n_{1},n_{2},n_{3}])(x)\\
&=[\alpha(n_{1}), \alpha(n_{2}), \alpha(n_{3})](x)\\
&=[\alpha(n_{1})(x), \alpha(n_{2})(x), \alpha(n_{3})(x)]\\
&=[\alpha_{x}(n_{1}), \alpha_{x}(n_{2}), \alpha_{x}(n_{3})],
\end{aligned}
\end{equation*}
\begin{equation*}
\begin{aligned}
\alpha_{x}(tn_{1})=\alpha(tn_{1})(x)=t\alpha(n_{1})(x)=t\alpha_{x}(n_{1})
\end{aligned}
\end{equation*}
for all $n_{1}, n_{2}, n_{3}\in N$, $t\in T$, $\alpha_{x}$ is a homomorphism of $T$-modules.\\
\indent Since $E$ is an injective $T$-module, for each $x\in X$ there exists a homomorphism of $T$-modules $\beta_{x}: M\longrightarrow E$ such that $\beta_{x}i=\alpha_{x}$ by Proposition~\ref{ge}, that is to say, the following diagram commutes\\
\[
\xymatrix @R=24pt{ N \ar[r]^-{i} \ar[d]_{\alpha_{x}} & M \ar@{-->}[dl]^-{{\beta_{x}}}
\\
 E. }
\]\\
\indent Define
\begin{equation*}
\begin{aligned}
 \beta: M &\longrightarrow E^{X}\\
       m &\longmapsto \beta(m):x\mapsto \beta_{x}(m),\\
\end{aligned}
\end{equation*}
where $m\in M$. It is easy to check that $\beta$ is a map.\\
Since\\
\begin{equation*}
\begin{aligned}
\beta([m_{1}, m_{2}, m_{3}])(x)&=\beta_{x}([m_{1}, m_{2}, m_{3}])\\
&=[\beta_{x}(m_{1}), \beta_{x}(m_{2}), \beta_{x}(m_{3})]\\
&=[\beta(m_{1})(x), \beta(m_{2})(x), \beta(m_{3})(x)]\\
&=[\beta(m_{1}), \beta(m_{2}), \beta(m_{3})](x),
\end{aligned}
\end{equation*}
\begin{equation*}
\begin{aligned}
\beta(tm_{1})(x)=\beta_{x}(tm_{1})=t\beta_{x}(m_{1})=t\beta(m_{1})(x)
\end{aligned}
\end{equation*}
for all $m_{1}, m_{2}, m_{3}\in M$, $t\in T$, $\beta$ is a homomorphism of $T$-modules.\\
\indent~Since for all $n\in N$ and $x\in X$
\begin{equation*}
\begin{aligned}
\beta(i(n))(x)=\beta_{x}(i(n))=\alpha_{x}(n)=\alpha(n)(x).
\end{aligned}
\end{equation*}
That is to say, the following diagram commutes\\
\[
\xymatrix @R=24pt{ N \ar[r]^-{i} \ar[d]_{\alpha} & M \ar@{-->}[dl]^-{{\beta}}
\\
 E^{X}. }
\]
Therefore $E^{X}$ is an injective $T$-module by Proposition~\ref{ge}.
\end{proof}
\end{proposition}

Following by the above proposition, $\mathbb{Q}^{\mathbb{Q}}$ is an injective T($\mathbb{Z})$-module.

\begin{theorem}(The Schanuel Lemma on Injective $T$-modules) Let $T$ be a truss. Suppose that the following two sequences of $T$-modules are exact.
$$\xymatrix @R=24pt{
\star \ar@{.>}[r] & M \ar@<+0.3ex>@{.>}[r]^{i} \ar@<-0.3ex>[r] & E \ar[r]^\pi & Q \ar[r] & \star}$$
$$\xymatrix @R=24pt{ \star \ar@{.>}[r] & M \ar@<+0.3ex>@{.>}[r]^{i^{\prime}} \ar@<-0.3ex>[r] & E^{\prime} \ar[r]^{\pi^{\prime}} & Q^{\prime} \ar[r] & \star}.$$\\
Moreover, \emph{Abs}$(Q^{\prime})$ is not empty, $E$ and $E^{\prime}$ are injective $T$-modules.~Then there exists an isomorphism $E\times Q^{\prime}\cong E^{\prime}\times Q$ as $T$-modules. Thus, $Q$ is an injective $T$-module if and if only $Q^{\prime}$ is an injective $T$-module.
\begin{proof} Considering the following commutative diagram:
$$\xymatrix @R=24pt{
\star \ar@{.>}[r] & M \ar@<+0.3ex>@{.>}[r]^{i} \ar@<-0.3ex>[r]& E \ar[r]^{\pi} & Q \ar@{>}[r]^{} & \star\\
\star \ar@{.>}[r] & M \ar@<+0.3ex>@{.>}[r]^{i^\prime} \ar@<-0.3ex>[r] \ar@{>}[u]^{1_{M}} & E^\prime \ar[r]^{\pi^\prime} \ar@{-->}[u]^{\alpha} & Q^\prime\ar@{>}[r]^{}\ar@{-->}[u]^{\beta} & \star}$$\\
\indent Since $E$ is an injective $T$-module, there exists a morphism of $T$-modules $\alpha$ : $E^\prime\longrightarrow E$ such that $\alpha i^\prime = i$. By diagram chasing, there exists a morphism of $T$-modules $\beta$ : $Q^\prime\rightarrow Q$ such that $\beta\pi^\prime = \pi\alpha$.\\
Define
\begin{equation*}
\begin{aligned}
\theta:E^\prime&\longrightarrow E\times Q^\prime\\
      x^\prime&\longmapsto (\alpha(x^\prime), \pi^\prime(x^\prime))\\
\end{aligned}
\end{equation*}
and
\begin{equation*}
\begin{aligned}
\psi:E\times Q^\prime&\longrightarrow Q\\
      (x, q^\prime)&\longmapsto [\pi(x), \beta(q^\prime), \beta(e)],\\
\end{aligned}
\end{equation*}
where $x^\prime\in E^\prime$, $x\in E$, $q^\prime\in Q^\prime$ and $e\in$ Abs$(Q^\prime)$. It is easy to verify that $\theta$ and $\psi$ are well-defined.\\
First of all, since
\begin{equation*}
\begin{aligned}
\theta([x_{1}^\prime, x_{2}^\prime, x_{3}^\prime])&=(\alpha[x_{1}^\prime, x_{2}^\prime, x_{3}^\prime], \pi^\prime[x_{1}^\prime, x_{2}^\prime, x_{3}^\prime])\\
&=([\alpha(x_{1}^\prime), \alpha(x_{2}^\prime), \alpha(x_{3}^\prime)], [\pi^\prime(x_{1}^\prime), \pi^\prime(x_{2}^\prime), \pi^\prime(x_{3}^\prime)])\\
&=[(\alpha(x_{1}^\prime), \pi^\prime(x_{1}^\prime)), (\alpha(x_{2}^\prime), \pi^\prime(x_{2}^\prime)), (\alpha(x_{3}^\prime), \pi^\prime(x_{3}^\prime))],\\
&=[\theta(x_{1}^\prime), \theta(x_{2}^\prime), \theta(x_{3}^\prime)],
\end{aligned}
\end{equation*}
\begin{equation*}
\begin{aligned}
\theta(tx_{4}^\prime) = (\alpha(tx_{4}^\prime), \pi^\prime(tx_{4}^\prime)) = (t\alpha(x_{4}^\prime), t\pi^\prime(x_{4}^\prime)) = t(\alpha(x_{4}^\prime), \pi^\prime(x_{4}^\prime)) = t\theta(x_{4}^\prime)
\end{aligned}
\end{equation*}
for all $x_{1}^\prime, x_{2}^\prime, x_{3}^\prime, x_{4}^\prime\in E^\prime, t\in T$.
And for all $(x_{1}, q_{1}^\prime), (x_{2}, q_{2}^\prime), (x_{3}, q_{3}^\prime), (x_{4}, q_{4}^\prime)\in E\times Q^\prime$, $t\in T$,
\begin{equation*}
\begin{aligned}
\psi([(x_{1}, q_{1}^\prime), (x_{2}, q_{2}^\prime), (x_{3}, q_{3}^\prime)])&=\psi([x_{1}, x_{2}, x_{3}], [q_{1}^\prime, q_{2}^\prime, q_{3}^\prime])=[\pi[x_{1}, x_{2}, x_{3}],  \beta[q_{1}^\prime, q_{2}^\prime, q_{3}^\prime], \beta(e)]\\
&=[[\pi(x_{1}), \pi(x_{2}), \pi(x_{3})], [\beta(q_{1}^\prime), \beta(q_{2}^\prime), \beta(q_{3}^\prime)], [\beta(e), \beta(e), \beta(e)]]\\
&\overset {Lemma~\ref{le.se}}{=}[[\pi(x_{1}), \beta(q_{1}^\prime), \beta(e)], [\pi(x_{2}), \beta(q_{2}^\prime), \beta(e)], [\pi(x_{3}), \beta(q_{3}^\prime), \beta(e)]]\\
&=[\psi(x_{1}, q_{1}^\prime), \psi(x_{2}, q_{2}^\prime), \psi(x_{3}, q_{3}^\prime)],
\end{aligned}
\end{equation*}
\begin{equation*}
\begin{aligned}
\psi(t(x_{4}, q_{4}^\prime))&= \psi(tx_{4}, tq_{4}^\prime)=[\pi(tx_{4}), \beta(tq_{4}^\prime), \beta(e)]\\
&=[\pi(tx_{4}), \beta(tq_{4}^\prime), \beta(te)]=[t\pi(x_{4}), t\beta(q_{4}^\prime), t\beta(e)]\\
&=t[\pi(x_{4}), \beta(q_{4}^\prime), \beta(e)]=t\psi(x_{4}, q_{4}^\prime).\\
\end{aligned}
\end{equation*}
Then $\theta$ and $\psi$ are homomorphisms of $T$-modules.\\
\indent Secondly, it suffices to show that $\theta$ is monic. Aussme that $(\alpha(x^\prime), \pi^\prime(x^\prime)) = (\alpha(y^\prime), \pi^\prime(y^\prime))$, then $\alpha(x^\prime) = \alpha(y^\prime)$, $\pi^\prime(x^\prime) = \pi^\prime(y^\prime)$. As
$\xymatrix @R=24pt{
\star \ar@{.>}[r] & M \ar@<+0.3ex>@{.>}[r]^{i^{\prime}} \ar@<-0.3ex>[r] & E^\prime \ar[r]^{\pi{^\prime}} & Q^\prime \ar[r] & \star}$ is exact, there exists $q^\prime\in Q^\prime$ such that Im$i^\prime$ = ker$_{q}\pi^\prime$. Since $\pi^\prime$ is epic, there exists $z^\prime\in E^\prime$ such that $\pi^\prime(z^\prime) = q^\prime$, this is means that $z^\prime\in$ ker$_{q^\prime}\pi^\prime$ = Im$i^\prime$. Since $i^\prime$ is monic, there exists a unique $m\in M$ such that $i^\prime(m) = z^\prime$. However,
$$\pi^\prime([x^\prime, y^\prime, z^\prime]) = [\pi^\prime(x^\prime), \pi^\prime(y^\prime), \pi^\prime(z^\prime)]\overset {Lemma~\ref{le.se}}{=}\pi^\prime(z^\prime) = q^\prime.$$
This implies that $[x^\prime, y^\prime, z^\prime]\in$ ker$_{q^\prime}\pi^\prime$ = Im$i^\prime$. Since $i^\prime$ is monic, there exists a unique $m^\prime\in M$ such that $i^\prime(m^\prime) = [x^\prime, y^\prime, z^\prime]$. Additional, since
$$i(m^\prime)\overset{i=\alpha i^{\prime}}{=}\alpha i^\prime(m^\prime) = \alpha([x^\prime, y^\prime, z^\prime]) = [\alpha(x^\prime), \alpha(y^\prime), \alpha(z^\prime)]\overset {Lemma~\ref{le.se}}{=}\alpha(z^\prime)$$ and $i(m)\overset{i=\alpha i^{\prime}}{=}\alpha i^\prime(m) = \alpha(z^\prime)$, then $i(m^\prime) = i(m)$. Since $i$ is monic, $m^\prime = m$. Thus, $i^\prime(m^\prime) = i^\prime(m) = z^\prime$, then $x^\prime=y^\prime$, by Lemma \ref{le.se}. Therefore, $\theta$ is monic.\\
\indent Thirdly, to prove that $\psi$ is epic. Since $\pi$ is epic, there exists $x\in E$ such that $\pi(x) = q_{1}$ for any $q_{1}\in Q$. Since $\psi(x, e) = [\pi(x), \beta(e), \beta(e)]\overset {Lemma~\ref{le.se}}{=}\pi(x) = q_{1}$, then proving that $\psi$ is epic.\\
\indent Fourthly, it suffices to show that
$$\xymatrix @R=24pt{
\star \ar@{.>}[r] & E^\prime \ar@<+0.3ex>@{.>}[r]^-{\theta} \ar@<-0.3ex>[r] & E\times Q^\prime\ar[r]^-\psi & Q\ar[r] & \star}$$
is exact. Since $\psi\theta(x^\prime) = \psi(\alpha(x^\prime),~\pi^\prime(x^\prime)) = [\pi\alpha(x^\prime), \beta\pi^\prime(x^\prime), \beta(e)]\overset {\beta\pi^{\prime}=\pi\alpha}{=}\beta(e)$ for any $x^\prime\in E^\prime$, that is to say, Im$\theta\subseteq$ ker$_{\beta(e)}\psi$. Let $(x, q^\prime)\in$ ker$_{\beta(e)}\psi$. So $\psi(x, q^\prime)=[\pi(x), \beta(q^\prime), \beta(e)]=\beta(e)$, then by Lemma~\ref{le.se}, $\pi(x)=\beta(q^\prime)$. Since $\pi^\prime$ is epic, there exists $x^\prime\in E^\prime$ such that $\pi^\prime(x^\prime) = q^\prime$. As $\xymatrix @R=24pt{
\star \ar@{.>}[r] & M \ar@<+0.3ex>@{.>}[r]^{i} \ar@<-0.3ex>[r] & E \ar[r]^\pi & Q \ar[r] & \star}$ is exact, there exists $q\in Q$ such that Im$i$ = ker$_{q}\pi$. So $\beta\pi^\prime(x^\prime)\overset {\beta\pi^{\prime}=\pi\alpha}{=}\pi\alpha(x^\prime)=\beta(q^\prime)=\pi(x)$. Since $\pi$ is epic, there exists $y\in E$ such that $\pi(y) = q$, this means that $y\in$ ker$_{q}\pi$ = Im$i$. Since $i$ is monic, there exists a unique $m\in M$ such that $i(m)=y\overset {\alpha i^{\prime}=i}{=}\alpha i^\prime(m)$. However,
$$\pi([x, \alpha(x^\prime), y]) = [\pi(x), \pi\alpha(x^\prime), \pi(y)]\overset {Lemma~\ref{le.se}}{=}\pi(y) = q.$$
This implies that $[x, \alpha(x^\prime), y]\in$ ker$_{q}\pi$ = Im$i$. Since $i$ is monic, there exists a unique $m^{\prime}\in M$ such that $i(m^\prime)\overset{\alpha i^{\prime}=i}{=}\alpha i^\prime(m^\prime)= [x, \alpha(x^\prime), y]$, then
$$x\overset {Lemma~\ref{le.se}}{=}[\alpha i^\prime(m^\prime), y, \alpha(x^\prime)] = [\alpha i^\prime(m^\prime), \alpha i^\prime(m), \alpha(x^\prime)]= \alpha([i^\prime(m^\prime), i^\prime(m), x^\prime]).$$
Since $[i^\prime(m^\prime), i^\prime(m), x^\prime]\in E^\prime$,
\begin{equation*}
\begin{aligned}
\theta([i^\prime(m^\prime), i^\prime(m), x^\prime])&=(\alpha([i^\prime(m^\prime), i^\prime(m), x^\prime]), \pi^\prime([i^\prime(m^\prime), i^\prime(m), x^\prime]))\\
&=(x, [\pi^\prime i^\prime(m^\prime), \pi^\prime i^\prime(m), \pi^\prime(x^\prime)])\\
&=(x, [q^\prime, q^\prime, q^\prime])\\
&=(x, q^\prime).
\end{aligned}
\end{equation*}
This implies that Im$\theta\supseteq$ ker$_{\beta(e)}\psi$, and hence Im$\theta$ = ker$_{\beta(e)}\psi$. Therefore, proving that the sequence is exact.\\
\indent Finally, to prove that  $\xymatrix @R=24pt{
\star \ar@{.>}[r] & E^\prime \ar@<+0.3ex>@{.>}[r]^-{\theta} \ar@<-0.3ex>[r] & E\times Q^\prime \ar[r]^-\psi & Q \ar[r] & \star}$~splits. Since $E^\prime$ is an injective $T$-module, there exists a morphism of $T$-modules $\gamma$: $E\times Q^\prime\longrightarrow E^\prime$ such that $\gamma\theta = 1_{E^\prime}$.
$$
\xymatrix @R=20pt{
\star \ar@{.>}[r] & E^\prime \ar@<+0.5ex>@{.>}[r]^-{\theta} \ar[r] & E\times Q^\prime \ar[r]^-{\psi} \ar@/^3ex/[l]^-{\gamma} & Q \ar[r] & \star
}
$$
By Proposition \ref{we}, $\xymatrix @R=24pt{
\star \ar@{.>}[r] & E^\prime \ar@<+0.3ex>@{.>}[r]^-{\theta} \ar@<-0.3ex>[r] & E\times Q^{\prime} \ar[r]^-\psi & Q \ar[r] & \star}$~splits, thus $E^\prime\times Q \cong E\times Q^\prime$ as $T$-modules. So $Q$ is an injective $T$-module if and if only $Q^\prime$ is an injective $T$-module by Proposition~\ref{pe}.
\end{proof}
\end{theorem}

\textbf{Yongduo Wang}\\
Department of Applied Mathematics, Lanzhou University of Technology, 730050 Lanzhou, Gansu, P. R. China\\
E-mail: \textsf{ydwang@lut.edu.cn}\\[0.3cm]
\textbf{Shujuan Han}\\
Department of Applied Mathematics, Lanzhou University of Technology, 730050 Lanzhou, Gansu, P. R. China\\
E-mail: \textsf{1690018130@QQ.com}\\[0.3cm]
\textbf{Dengke Jia}\\
Department of Applied Mathematics, Lanzhou University of Technology, 730050 Lanzhou, Gansu, P. R. China\\
E-mail: \textsf{1719768487@QQ.com}\\[0.3cm]
\textbf{Jian He}\\
Department of Applied Mathematics, Lanzhou University of Technology, 730050 Lanzhou, Gansu, P. R. China\\
E-mail: \textsf{jianhe30@163.com}\\[0.3cm]
\textbf{Dejun Wu}\\
Department of Applied Mathematics, Lanzhou University of Technology, 730050 Lanzhou, Gansu, P. R. China\\
E-mail: \textsf{wudj@lut.edu.cn}\\[0.3cm]
\end{document}